\documentclass[11pt]{article}
\usepackage{blindtext}
\usepackage{enumitem} 
\setlist[enumerate]{parsep=0pt}
 
\usepackage{ragged2e} 

\usepackage{mathrsfs}
\usepackage{amssymb}
\usepackage{layout}
\usepackage{graphicx}
\usepackage{marvosym}
\usepackage[usenames,dvipsnames]{color}
\usepackage{verbatim}
\usepackage{enumitem}
\usepackage[pdftex]{hyperref}
\usepackage{fancyhdr}
\usepackage{indentfirst}
\usepackage{centernot}
\usepackage[all]{xy}

\usepackage{amsthm,etoolbox}

\makeatletter
\@addtoreset{theorem}{section}
\@addtoreset{theorem}{subsection}
\@addtoreset{theorem}{subsubsection}

\usepackage{amsmath}

\usepackage{tikz-cd}
\usetikzlibrary{calc}

\usepackage{hyperref}
\usepackage{cleveref}[noabbrev]
\newtheorem{theorem}{Theorem}

\theoremstyle{definition}

\newtheorem{definition}[theorem]{Definition}

\theoremstyle{theorem}

\newtheorem{lemma}[theorem]{Lemma}
\newtheorem{prop}[theorem]{Proposition}
\newtheorem{cor}[theorem]{Corollary}

\newtheorem{remark}[theorem]{Remark}

\crefname{theorem}{Theorem}{Theorems}
\crefname{lemma}{Lemma}{Lemmas}
\crefname{prop}{Proposition}{Propositions}
\crefname{fact}{Fact}{Facts}
\crefname{remark}{Remark}{Remarks}
\crefname{cor}{Corollary}{Corollaries}

\newcommand{\theoremprefix}{}
\let\thetheoremsaved\thetheorem
\renewcommand{\thetheorem}{\theoremprefix\thetheoremsaved}

\patchcmd{\@startsection}{\par}{\renewcommand{\theoremprefix}{\csname the#1\endcsname.}}{}{}

\makeatother

\begin{document}

\clearpage
\pagenumbering{arabic}

\def\tp{\mbox{\rm tp}}
\def\qftp{\mbox{\rm qftp}}
\def\cb{\mbox{\rm cb}}
\def\wcb{\mbox{\rm wcb}}
\def\Diag{\mbox{\rm Diag}}
\def\trdeg{\mbox{\rm trdeg}}
\def\Gal{\mbox{\rm Gal}}
\def\Lin{\mbox{\rm Lin}}

\def\restriction#1#2{\mathchoice
              {\setbox1\hbox{${\displaystyle #1}_{\scriptstyle #2}$}
              \restrictionaux{#1}{#2}}
              {\setbox1\hbox{${\textstyle #1}_{\scriptstyle #2}$}
              \restrictionaux{#1}{#2}}
              {\setbox1\hbox{${\scriptstyle #1}_{\scriptscriptstyle #2}$}
              \restrictionaux{#1}{#2}}
              {\setbox1\hbox{${\scriptscriptstyle #1}_{\scriptscriptstyle #2}$}
              \restrictionaux{#1}{#2}}}
\def\restrictionaux#1#2{{#1\,\smash{\vrule height .8\ht1 depth .85\dp1}}_{\,#2}} 

\newcommand{\forkindep}[1][]{%
  \mathrel{
    \mathop{
      \vcenter{
        \hbox{\oalign{\noalign{\kern-.3ex}\hfil$\vert$\hfil\cr
              \noalign{\kern-.7ex}
              $\smile$\cr\noalign{\kern-.3ex}}}
      }
    }\displaylimits_{#1}
  }
}
\newpage

\begin{center}

\large \MakeUppercase{A note on some examples of NSOP1 theories}

\vspace{5mm}

    \large Yvon \textsc{Bossut}
\end{center}
\vspace{10mm}

\vspace{10pt}
Abstract : We present here some known and some new examples of non-simple NSOP1 theories and some behavior that Kim-forking can exhibit in these theories, in particular that Kim-forking after forcing base monotonicity can or can not satisfy extension (on arbitrary sets).  This study is based on the results of Chernikov, Ramsey, Dobrowolski and Granger.\footnote{Partially supported by ANR-DFG AAPG2019 GeoMod}

\section{{\Large Context and notations :}}

NSOP1 theories have recently been studied as a generalization of simple theories. Kim and Pillay \cite{kim1997simple} have shown that simplicity can be characterized in terms of the existence of an independence notion satisfying some properties, which then implies that this independence relation is forking independence. Chernikov and Ramsey \cite[Theorem 5.8]{chernikov2016model} have shown a similar result for NSOP1 theories and Kim-forking independence, another relation of independence which is defined as 'generic forking independence'. One question we can ask is the relation between these two notions of independence in the context of NSOP1 theories.

\vspace{10pt}
In this note we will present some known examples of NSOP1 theories as well as generalizations of these examples and we will compute, whenever possible, forking independence and Kim-independence in these theories. We will show that the relation between these two notions can vary among non-simple NSOP1 theories, and we will also show that Kim-forking may or may not satisfy strong local character.

\vspace{10pt}
We shall write $\forkindep^{d}$ for dividing independence, $\forkindep^{f}$ for forking independence and $\forkindep^{K}$ for Kim-independence. By algebraic independence we mean $A\forkindep^{a}_{C}B:=acl(AC)\cap acl(BC)=acl(C)$. The main notions we shall be looking at are the two following definitions which were introduced by H. Adler in \cite{adler2009geometric} in his axiomatic approach of independence relations. They were also studied in the NSOP1 context and also more generally by C. D'Elbée in \cite{d2023axiomatic} and \cite{d2022generic}.

\begin{definition} Forcing base monotonicity : Let $\forkindep$ be an independence relation defined on algebraically closed sets $A,B,C$ such that $C\subseteq A,B$. We define $\forkindep^{M}$ as the weakest independence relation that implies $\forkindep$ and satisfies base monotonicity, algebraic closure and normality, meaning : $A\forkindep^{M}_{C}B$ iff $acl(AB')\forkindep^{}_{acl(CB')}B$ for all $B'\subseteq acl(CB)$.
\end{definition}

\begin{definition} Forcing extension : Let $\forkindep$ be an independence relation defined on algebraically closed sets $A,B,C$ such that $C\subseteq A,B$. We define $\forkindep^{*}$ as the weakest independence relation that implies $\forkindep$ and satisfies extension, meaning : $A\forkindep^{*}_{C}B$ iff for all $B'\supseteq B$ there exist $A'\equiv_{B} A$ such that $A\forkindep_{C}B'$.
\end{definition}



This work is dedicated to the study the two sorted theory of bilinear forms. Let $T_{K}$ be the theory of a field, we write $sT^{K}_{\infty}$ the theory of a vector space of infinite dimension over a model of $T_{K}$ with a non-degenerate symmetric bilinear form. We will study these theories and show different results depending on $T_{K}$.

\vspace{10pt}
We define $\mathcal{L}_{0}$ as the language of rings on the field sort $K$, the sum and the vector $0_{V}$ on the vector sort $V$ and the scalar product $\cdot : K\times V \rightarrow V$. We define the theta functions $\theta_{n}: V^{n+1} \rightarrow K^{n} $ for any $n<\omega$ which is defined the following way: It sends any tuple of vectors $\overline{v}=(v_{i})_{i\leq n}$ with $v_{<n}$ linearly independent and $v_{n}$ in the linear span of $v_{<n}$ to the unique tuple of scalars $(\lambda_{i})_{i<n}$ such that $v_{n} = \sum_{i<n} \lambda_{i}\cdot v_{i}$, and any other tuple to $(0_{K},..,0_{K})$. We will work with the language $\mathcal{L}= \mathcal{L}_{0} \cup \lbrace \theta_{n}$ : $n<\omega \rbrace$.

\vspace{10pt}
The theta functions are definable with quantifiers in $\mathcal{L}_{0}$, the point of adding them into the language is to have Quantifier Elimination in $\mathcal{L}$. For two substructures $A$ and $B$ of some model $M$ of $sT^{K}_{\infty}$ we will write $\langle A \rangle$ for the linear span of $V(A)$, and use the concatenation $AB$ to denote the structure generated by $A$ and $B$.

\section{Generic bilinear form on an NSOP$_1$ field}

We generalize the characterization of Kim-independence over models in $sT^{ACF}_{\infty}$ to the case of NSOP$_1$ fields. To have the most generality as possible on the theory of the field we add an additional axiom to $sT^{K}_{\infty}$ so that this theory is NSOP$_1$ if and only if the theory $T_{K}$ of the field is NSOP$_1$, and compute what Kim-forking is over models, and over algebraically closed set in the case when we assume that $sT^{K}_{\infty}$ satisfies existence.

\subsection{NSOP$_1$ and Kim-independence}

Let us consider an NSOP$_1$ theory of fields $T_{K}$ with quantifier elimination in the language $\mathcal{L}_{K}$ and let $sT^{K}_{\infty }$ be the two sorted theory of infinite dimensional vector spaces over a model of $T_{K}$ with a non degenerate symmetric bilinear form (written $[$ $,$ $]$). For a set of parameter $A$ we will write $A^{\perp}$ to mean $\lbrace u$ : $[u,x]=0$ for all $x\in V(A)\rbrace$. We will also assume the following: 

\begin{center}
$(\dagger)$: if $\overline{b}$ is a finite tuple there is $u\in \overline{b}^{\perp}\setminus \lbrace 0 \rbrace$ such that $[u,u]=0$.\end{center}

The condition $(\dagger)$ is a sufficient condition for this theory to have Q.E. $(\dagger)$ is in particular satisfied when every element of the field is a square. Let $\mathbb{M} = (\mathbb{V},\mathbb{K})$ be a monster model of $sT^{K}_{\infty }$.

\begin{lemma} \label{algbili}
    Let $(\lambda_{i})_{i<n}$ be a tuple of scalars, $(b_{i})_{i<n}$ be linearly independent vectors in $\mathbb{M}$, $\alpha \in \mathbb{K}$ and $\overline{e}$ be a finite tuple of vectors. Then there is $x\in \mathbb{V}$ such that $[x,x]=\alpha$, $[x,b_{i}]=\lambda_{i}$ for all $i<n$ and $x \centernot\in \langle \overline{e} \rangle$.
\end{lemma}

\begin{proof}
We begin by showing that there is $x\in \mathbb{V}$ such that $[x,b_{i}]=\lambda_{i}$ for all $i<n$ and $x \centernot\in \langle \overline{e}\rangle$, for this we do not need the assumption ($\dagger$). What we want to show is that given a linearly independent tuple $\overline{b}=b_{0},..,b_{n}$ the linear function $\psi_{\overline{b}}$: $x\longrightarrow ([x,b_{i}])_{i\leq n}$ is surjective, the other condition follows then from the fact that its kernel has infinite dimension.

\vspace{10pt}
Now let us prove surjectivity by induction on $n<\omega$. For $n=0$ it is trivial, now assume that it holds for $n>0$. We want to find $b'_{0} \in (b_{1},..,b_{n})^{\perp}\setminus b_{0}^{\perp}$. By assumption there is $b'_{i}$ for $1\leq i\leq n$ such that $[b'_{i},b_{j}]=\delta_{i,j}$ for all $1\leq i,j\leq n$. With this every $u\in \mathbb{V}$ can be written:

\begin{center}
$u=\sum\limits_{1\leq i\leq n}[u,b_{i}]\cdot b'_{i} + (u-\sum\limits_{1\leq i\leq n}[u,b_{i}]\cdot b'_{i})$.
\end{center}

So $\mathbb{V} = \bigoplus\limits_{1\leq i\leq n} \mathbb{K}\cdot b'_{i} + (b_{1},..,b_{n})^{\perp}$, we will write $u_{0} = u-\sum\limits_{1\leq i\leq n}[u,b_{i}]\cdot b'_{i}$.

\vspace{10pt}
Now if such a $b'_{0}$ does not exist we would have that $(b_{1},..,b_{n})^{\perp} \subseteq b_{0}^{\perp}$, using the previous decomposition, for all $u\in \mathbb{V}$ we have:

$$[u,b_{0}]=\sum\limits_{1\leq i\leq n}[b'_{i},b_{0}]\cdot [u,b_{i}] + [u_{0},b_{0}]=\sum\limits_{1\leq i\leq n}[b_{i},b_{0}]\cdot [u,b_{i}],$$

which by non-degeneracy contradicts the fact that $b_{0} \centernot\in \langle b_{1},..,b_{n}\rangle$.

\vspace{10pt}

We now use our assumption $(\dagger)$ to choose the value of the quadratic form. Choose $x\in \mathbb{V}$ such that $[x,b_{i}]=\lambda_{i}$ for all $i<n$ and $x \centernot\in \langle \overline{e} \rangle$. By what we have just shown and $(\dagger)$ there are $x',x''$ linearly independent such that $x',x'' \centernot\in \langle \overline{e},x\rangle$, $x',x''\in (\overline{b},x)^{\perp}$, $[x',x']=0$ and $[x',x'']=1$. Then the element $x+(\alpha -[x,x] - [x'',x''])\cdot x' + x''$ corresponds to what we are looking for.

\vspace{10pt}

We now show that when every element of the field is a square the condition $(\dagger)$ is satisfied. Let $\overline{b}$ be a finite tuple of vectors and $x\in \overline{b}^{\perp}\setminus \langle \overline{b} \rangle$. By the previous point we can find $y\in \overline{b}^{\perp}\setminus \langle \overline{b}x \rangle$ such that $[x,y]=0$. Then, if $[x,x]=0$ or $[y,y]=0$ we have the element we are looking for, and else if $\lambda^{2}=[x,x]$ and $\mu^{2}=-[y,y]$ then $\mu \cdot x + \lambda \cdot y$ satisfies our conditions.
\end{proof}

\begin{cor}
    The theory $sT^{K}_{\infty }$ has Q.E. in the language $\mathcal{L}=\mathcal{L}_{K}\cup \lbrace [$ $,$ $],\cdot, + \rbrace \cup \lbrace \theta_{n}$ : $ n<\omega\rbrace$. This implies in particular that the field is stably embedded.
\end{cor}
\begin{proof}
For this we show that finite partial isomorphism inside of $\mathbb{M}$ have the back and forth property. Consider two finitely generated structures $A$ and $B$ in $\mathbb{M}$ such that $\varphi$: $A \rightarrow B$ is an isomorphism. We show that we can extend it to any element $a\in\mathbb{M}$.
    
\vspace{10pt}    
Let $a_{1},..,a_{n-1}$ be a base of $V(A)$ as a $K(A)$ vector space and $b_{i}:=\varphi(a_{i})$, which is then a base of $V(B)$ as a $K(B)$ vector space. Let $\varphi_{K}$ : $K(A)\rightarrow K(B)$ be the induced field isomorphism. By quantifier elimination in the field sort we can extend $\varphi_{K}$ to any element of $\mathbb{K}$, so if $a\in \mathbb{K}$ we can extend $\varphi$ to $Aa$. If $a\in \langle A \rangle$ and $a= \sum\limits_{i} \lambda_{i}\cdot a_{i}$ it is clear that extending $\varphi$ to the structure generated by $Aa$ is equivalent to extending $\varphi$ to $A(\lambda_{i})_{i<n}$, which we can do by quantifier elimination in the field sort.

\vspace{10pt}
Now if $a\centernot\in \langle A \rangle$ let us consider an extension $\varphi_{K}'$ of $\varphi_{K}$ to the field sort of the structure generated by $Aa$ (which is just $K(A)([a,a],[a,a_{i}])_{i<n}$). Now using \cref{algbili} we can find $b\centernot\in \langle B \rangle$ such that $[b,b]=\varphi'_{K}$ and $[b,b_{i}]=\varphi_{K}'([a,a_{i}])$ for all $i<n$. It is clear that sending $a$ to $b$ extends the isomorphism $\varphi$.\end{proof}

Without the assumption $(\dagger)$ quantifier elimination in the theory of infinite dimensional vector spaces over a model of $T_{K}$ with a non degenerate symmetric bilinear form in the language $\mathcal{L}$ depends on the theory of the field (having squares for example) and also on the bilinear form (being positively defined in a RCF for example), the ($\dagger$) assumption allows us to keep full generality on the theory of the field.

\begin{definition}
For $E \subseteq A,B$ algebraically closed subsets of $\mathbb{M}$ and $E\models sT^{K}_{\infty }$, we define the relation $A\forkindep^{\infty}_{E}B$ by: $\langle A \rangle \cap \langle B \rangle = \langle E \rangle$ and $K(A)\forkindep^{K}_{K(E)}K(B)$ in the field theory.
\end{definition}

\begin{prop}
$sT^{K}_{\infty }$ is NSOP$_1$ and $\forkindep^{\infty} = \forkindep^{K}$ over models.
\end{prop}

\begin{proof}
We will now show that this relation satisfies the conditions of the first part of the Kim-Pillay theorem for Kim-forking over models \cite[Proposition 5.8]{chernikov2016model}.

\vspace{10pt}
Clearly it satisfies symmetry, monotonicity and existence over models. For strong finite character if $A\centernot\forkindep^{\infty}_{E}B$ and $K(A)\centernot\forkindep^{K}_{K(E)}K(B)$ then we can take the formula in $\tp(A/B)$ that Kim-forks over $E$. If $\langle A \rangle \cap \langle B \rangle \centernot= \langle E \rangle$ let $(e_{i})_{i<\alpha}$ be a base of $V(E)$ (it is then also a base of $\langle E \rangle$), $(a_{i})_{i<\beta}$ completing it into a base of $V(A)$ and $(b_{i})_{i<\beta}$ completing it into a base of $V(B)$. Then we have finite tuples $\overline{e}$, $\overline{a}$ and $\overline{b}$ in these bases and an element $a\in \langle \overline{a} \rangle)$ such that $a \in \langle \overline{e}\overline{b}\rangle \setminus \langle \overline{e} \rangle$, and the formula expressing this satisfies the condition.

\vspace{10pt}
We show that $\forkindep^{\infty}$ satisfies the independence theorem over models. Let $E \subseteq A_{0},A_{1},B_{0},B_{1}$ with $E$ a model, $A_{0}\equiv_{E}A_{1}$, $A_{j} \forkindep_{E}^{\infty}B_{j}$ for $j=0,1$ and $B_{0}\forkindep_{E}^{\infty}B_{1}$. Let $(e_{i})_{i<\alpha}$ be a base of $V(E)$, complete it by $(b^{j}_{i})_{i<\beta_{j}}$ into a base of $V(B_{j})$ for $j=0,1$, so $(e_{i})_{i<\alpha}(b^{0}_{i})_{i<\beta_{0}}(b^{1}_{i})_{i<\beta_{1}}$ is a base of $B_{0}B_{1}$ since $B_{0}$ and $B_{1}$ are linearly independent on $E$. 

\vspace{10pt}
Let $(a^{0}_{i})_{i<\gamma}$ complete $(e_{i})_{i<\alpha}$ into a base of $A_{0}$. Let $a_{0}=k_{0}v_{0}$ be a tuple enumerating $A_{0}$, $A_{0}\equiv_{E}A_{1}$ so we can find a tuple $a_{1}=k_{1}v_{1}$ enumerating $A_{1}$ such that $\varphi : a_{0} \longrightarrow a_{1}$ is an $E$-elementary embedding. Define $(a^{1}_{i})_{i<\gamma}:=\varphi((e_{i})_{i<\alpha})$. Since $\varphi$ is elementary over $E$ we have $\varphi([e_{i},a^{0}_{j}])=[e_{i},a^{1}_{j}]$ and $\varphi([a^{0}_{i},a^{0}_{j}])=[a^{1}_{i},a^{1}_{j}]$ for all $i,j$.

\vspace{10pt}
We have that $\tp(k_{0}/K(E))=\tp(k_{1}/K(E))$, $k_{j} \forkindep_{K(E)}^{K}K(B_{j})$ for $j=0,1$ and $K(B_{0})\forkindep_{K(E)}^{K}K(B_{1})$.

\vspace{10pt}
By applying the independence theorem for Kim-forking in $T_{K}$ we find $k\models \tp(k_{0}/K(B_{0})) \cup \tp(k_{1}/K(B_{1}))$ such that $k\forkindep_{K(E)}^{K}K(B_{0})K(B_{1})$. Using extension we can assume that $$k\forkindep_{K(E)}^{K}K(B_{0}B_{1}).$$

We have two partial embeddings $\varphi_{0} : k_{0} \longrightarrow k$ and $\varphi_{1} : k_{1} \longrightarrow k$ which are elementary over $K(B_{0})$ and $K(B_{1})$ respectively, by construction we have that $\varphi_{0} = \varphi_{1} \circ \varphi$. We will now define our structure $A$. Let $(a_{i})_{i<\gamma}$ be a tuple of linearly independent vectors not in $\langle B_{0}B_{1} \rangle$ such that:
    \begin{enumerate}
        \item[] $[e_{i},a_{j}]= \varphi_{0}([e_{i},a^{0}_{j}])=\varphi_{1}([e_{i},a^{1}_{j}])$ for all $i<\alpha,j<\gamma$,
        \item[] $[a_{i},a_{j}]= \varphi_{0}([a^{0}_{i},a^{0}_{j}])=\varphi_{1}([a^{1}_{i},a^{1}_{j}])$ for all $i,j<\gamma$,
       \item[] $[a_{i},b^{0}_{j}]= \varphi_{0}([a^{0}_{i},b^{0}_{j}])$ for all $i<\gamma,j<\beta_{0}$,
       \item[] $[a_{i},b^{1}_{j}]= \varphi_{1}([a^{1}_{i},b^{1}_{j}])$ for all $i<\gamma,j<\beta_{1}$.
    \end{enumerate}
    
This is possible by \cref{algbili} and compactness. Then the structure $A:=(\Lin_{k}((e_{i})_{i<\alpha}(a_{i})_{i<\gamma}),k)$ satisfies $A\forkindep^{\infty}_{E}B_{0}B_{1}$ and $A\equiv_{B_{j}}A_{j}$ for $j=0,1$ by quantifier elimination.

\vspace{10pt}
By \cite[Proposition 5.8]{chernikov2016model} we have that $sT^{K}_{\infty }$ is NSOP$_1$ and that$\forkindep^{\infty} \implies \forkindep^{K}$ over models. For the other implication, it is clear that if $A\forkindep^{K}_{E}B$ then $K(A)\forkindep^{K}_{K(E)}K(B)$: 

\vspace{10pt}
In fact if $A\forkindep^{K}_{E}B$ there is a sequence $(A_{i})_{i<\omega}$ which is coheir Morley over $E$ and $B$ indiscernible. Then the sequence $(K(A_{i}))_{i}$ is coheir Morley over $K(E)$ and $K(B)$ indiscernible, which by Kim's Lemma for Kim-forking in $T_{K}$ implies that $K(A)\forkindep^{K}_{K(E)}K(B)$.

\vspace{10pt}
Now if $\langle A \rangle \cap \langle B\rangle  \centernot= \langle E\rangle$, let $a\in \langle \overline{a}\rangle$ such that $a\in \langle B\rangle \setminus \langle E\rangle $ with $\overline{a}$ a finite tuple in $A$. Let $(B_{i})_{i<\omega}$ be a coheir Morley sequence over $E$ with $B_{0}=B$. We have $\langle B_{i}\rangle \cap \langle B_{<i}\rangle = \langle E\rangle $, so if $q(x,y):=\tp(\overline{a},B)$ then $\bigcup\limits_{i<\omega} q(x,B_{i})$ is inconsistent and $A\centernot\forkindep^{K}_{E}B$. We will give a more general argument when generalizing this to algebraically closed sets in \cref{forkindepouais}.\end{proof}

In particular if $T_{K}$ is a theory of fields, the theory $sT_{\infty}^{K}$ is NSOP$_1$ if and only if the theory $T_{K}$ is, and $sT_{\infty}^{K}$ is also always non-simple, as we can construct infinite decreasing sequences of type definable groups of unbounded index using orthogonality.

\subsection{More properties of independence relations}

We now assume that the theory of the field $T_{K}$ has existence. $\forkindep^{\infty}$ is then defined on arbitrary sets. We show that in this case some properties passes from $\forkindep^{K}$ in $T_{K}$ to $\forkindep^{\infty}$.

\begin{prop} If Kim-forking in $T_{K}$ satisfies strong local character, i.e. if for all finite tuple $\overline{a} \in \mathbb{K}$ and for all $B\subseteq \mathbb{K}$ small set there is a finite $B_{0}$ such that $a\forkindep^{K}_{B_{0}}B$, then $\forkindep^{\infty}$ also does. This is the case when $T_{K}$ is supersimple for example.
\end{prop}

\begin{proof}
Let $A$ be a finitely generated structure and  let $\overline{a}$ be a basis of $V(A)$ (the tuple $\overline{a}$ is then finite). Let $B\subseteq \mathbb{M}$ be a substructure, with $(b_{i})_{i<\beta}$ a $K(B)$-basis of $V(B)$. Since $\overline{a}$ is a finite tuple we can find a finite $\overline{b}\subseteq \lbrace b_{i}$ : $i<\beta \rbrace$ such that $\langle \overline{a}\rangle \cap \langle (b_{i})_{i<\beta} \rangle \subseteq \langle \overline{b} \rangle$. We then have that $\langle \overline{a}\overline{b} \rangle \cap \langle (b_{i})_{i<\beta} \rangle = \langle \overline{b} \rangle$.

\vspace{10pt}
We now consider the structure $A_{0}$ generated by $A\overline{b}$. We have $\langle A_{0} \rangle = \langle \overline{a}\overline{b}\rangle$, and the field sort is finitely generated, since $K(A)$ is finitely generated and by adding $\overline{b}$ we only have to add a finite number of images by the bilinear form and coefficients from linearly dependent tuples.

\vspace{10pt}
By assumption there is a finitely generated $E_{0}\subseteq K(B)$ such that $K(A_{0})\forkindep^{K}_{E_{0}}B$, and by setting $B_{0}$ to be the substructure of $B$ generated by $(\overline{b},E_{0})$ we get that $A\forkindep^{\infty}_{B_{0}}B$.
\end{proof}

\begin{lemma} \label{lindisj} $A\forkindep^{d}_{E}B$ implies that $\langle A \rangle \cap \langle B \rangle = \langle E \rangle$ for all $E\subseteq A,B$.
\end{lemma}

\begin{proof}
We can assume that $A$ is a finite tuple. As previously let $(e_{i})_{i<\alpha}$ be a basis of $V(E)$ and let $(e_{i})_{i<\alpha}(b_{i})_{i<\beta}$ a basis of $V(B)$.

\vspace{10pt}
We consider a tuple $(b^{j}_{i})_{i<\beta,j<\omega}$ linearly independent over $E$ such that:

\begin{enumerate}
    \item $[b_{i}^{j},e_{l}]=[b_{i},e_{l}]$ for all $i<\beta,j<\omega$,
    \item $[b_{i}^{j},b^{j}_{l}] = [b_{i},b_{l}]$ for all $i,l<\beta,j<\omega$,
    \item $[b_{i}^{j},b^{j'}_{l}] = 0$ for all $i,l<\beta$ and $j\centernot= j'<\omega$.
\end{enumerate}

We can find such a sequence by \cref{algbili} and compactness. By quantifier elimination it is an indiscernible sequence over $E$, so $\langle A \rangle \cap \langle B \rangle \subseteq \langle E \rangle$: Else the type of $A$ over $B_{0}$ would divides along that sequence since $\langle B_{i}\rangle \cap \langle B_{<i} \rangle = \langle E \rangle$ for all $i<\omega$.
\end{proof}

\begin{prop} Let $A$ be a structure such that $\forkindep^{K}$ in $T_{K}$ satisfies the independence theorem over $K(A)$. If $D^{0},D^{1}$ and $B^{0},B^{1}$ are such that $B^{0}\forkindep^{\infty}_{A}B^{1}$, $D^{0}\equiv_{A}D^{1}$ and $D^{j}\forkindep^{\infty}_{A}B^{j}$ for $j=0,1$ then there is $D'$ such that $D'\equiv_{AB^{j}}D^{j}$ for $j=0,1$ and $D'\forkindep^{\infty}_{A}B^{0}B^{1}$.
\end{prop}

\begin{proof} The proof is strictly similar to the one of amalgamation over models.
\end{proof}

We will now additionally assume that the theory $sT_{\infty}^{K}$ has existence, which was only proven for $ACF$. This assumption allows us to use Kim-forking independence over arbitrary sets and its properties, here we use Kim's lemma for Kim-forking.

\begin{prop}\label{forkindepouais} If $E\subseteq A,B$ are algebraically closed, then $A\forkindep^{K}_{E}B$ if and only if $A\forkindep^{\infty}_{E}B$.
\end{prop}
\begin{proof}
$[\Rightarrow] $: If $\langle A \rangle \cap \langle B \rangle \centernot= \langle E \rangle$, let $B_{i}$ be a Morley sequence with $B_{0}=B$. By \cref{lindisj} we have that $\langle B_{i} \rangle \cap \langle B_{<i}\rangle = \langle E \rangle$ for all $i<\omega$, so the type of $A$ over $B$ is inconsistent along that sequence, and so it Kim-divides over $E$. Now if $K(A)\centernot\forkindep^{K}_{K(E)}K(B)$, the induced structure in the field sort is the one of $T_{K}$ since the field is stably embedded, so $A\centernot\forkindep^{K}_{E}B$.

\vspace{10pt}

$[\Leftarrow]$: Let $A\forkindep^{\infty}_{E}B$, we show that $A\forkindep^{K}_{E}B$ using Kim's Lemma for Kim-forking.

\vspace{10pt}
Let $(e_{i})_{i<\kappa_{E}}$ be a basis of $V(E)$, complete it by $(a_{i})_{i<\kappa_{A}}$ into a basis of $V(A)$ and by $(b^{0}_{i})_{i<\kappa_{B}}$ into a basis of $V(B)$. Let $(B_{i})_{i<\omega}$ be an $E$-Morley sequence with $B_{0}=B$. Let $(b^{j}_{i})_{i<\kappa_{B}}$ be the basis of $B_{i}$ over $V(E)$ corresponding to $(b^{0}_{i})_{i<\kappa_{B}}$ for all $j<\omega$. 

\vspace{10pt}
Then the sequence $(K(B_{j}))_{j<\omega}$ is Morley over $K(E)$ and $K(A)\forkindep^{K}_{K(E)}K(B_{0})$. Let $k_{a}$ be a tuple enumerating $K(A)$ and let $q(x,y):=\tp(k_{a}K(B_{0}))$. Using the chain condition for Kim-forking we can find $k'\models \bigcup\limits_{j<\omega}q(x,K(B_{i}))$.

\vspace{10pt}
As for amalgamation over models there is a partial elementary embedding $\varphi$: $k_{a} \longrightarrow k'$ over $E$ such that for all $j<\omega$ this embedding extends into $\varphi_{i}$: $k_{a}K(B) \longrightarrow k'K(B_{i})$, which is also elementary over $K(E)$.

\vspace{10pt}
Let us now define a structure $A'$. Using \cref{algbili} we can find a tuple $(a'_{i})_{i<\gamma}$ of linearly independent vectors outside of $\langle B_{<\omega} \rangle$ such that:
\begin{enumerate}
        \item[] $[e_{i},a'_{j}]= \varphi_{0}([e_{i},a^{0}_{j}])=\varphi_{l}([e_{i},a_{j}])$ for all $i<\kappa_{E},j<\kappa_{A}$ and $l<\gamma$,
        \item[] $[a'_{i},a'_{j}]= \varphi_{0}([a_{i},a_{j}])=\varphi_{l}([a_{i},a_{j}])$ for all $i,j<\kappa_{A}$ and $l<\gamma$,
       \item[] $[a'_{i},b^{l}_{j}]= \varphi_{l}([a_{i},b_{j}])$ for all $i<\kappa_{A},j<\kappa_{B}$ and $l<\gamma$.
\end{enumerate}
    
By quantifier elimination the structure $A'=(\Lin_{k'}((e_{i})_{i<\alpha}(a'_{i})_{i<\gamma}),k')$ is such that $A'B_{j}\equiv_{E}AB$ for all $j<\omega$, so $A\forkindep^{K}_{E}B$ by Kim's Lemma for Kim-independence.
\end{proof}

\section{Proving existence in the case of simple fields}

We will now assume that $T_{K}$ is a simple theory of fields (see \cite{wagner2000simple} for properties of simple fields). We will generalize the notion of Gamma-forking that Granger developed in \cite{granger1999stability} for the case of $ACF$ and also give complete proofs. We use the same terminology as previously, and to lighten the notations we will write $A$ for the structure generated by the set $A$. As previously we work in a monster model of $sT^{K}_{\infty}$. We begin by recalling the properties of $f$-generics in groups definable in simple theories.

\begin{definition}
Let $(G,\cdot)$ be a group definable over $E$ in a model $M$ of some complete theory. For $E\subseteq B$ a set of parameters we say that an element $a\in G$ is \emph{f-generic in $G$ over $B$} if $a\cdot b\forkindep_{B}^{f}b$ for every $b\in G$ such that $a\forkindep^{f}_{B}b$.
\end{definition}

\begin{prop}\cite[Proposition 4.1.7]{wagner2000simple} Let $(G,\cdot)$ be a group definable over $E$ in a model $M$ of some simple theory and $B\supseteq E$. Then there is a generic type for $G$ over $B$.
\end{prop}

\begin{prop}\cite[Proposition 4.1.7]{wagner2000simple} Let $(G,\cdot)$ be a group definable over $E$ in a model $M$ of some simple theory and $B\supseteq E$. Then there is a generic type for $G$ over $B$.
\end{prop}

\begin{lemma}\cite[Lemma 4.1.2]{wagner2000simple}\label{fgener} Let $g$ be $f$-generic for $G$ over $A$ and let $A\subseteq B$. If $g\forkindep_{A}^{f}B$ then $g$ is $f$-generic over $B$.
\end{lemma}

We say that a tuple $(x_{i})_{i<\kappa} \in \mathbb{K}$ is $f$-generic over $K\subseteq \mathbb{K}$ if $x_{i}$ is $f$-generic over $\emptyset$ for all $i<n$ and $x_{i}\forkindep^{f}_{\emptyset}Kx_{\centernot =i}$. By \cref{fgener} this implies that $x_{i}$ is $f$-generic over $Kx_{\centernot= i}$.

\vspace{10pt}

We recall that in a simple field the notion of genericity for the additive group and the multiplicative group coincide. These notions coincide with genericity for affine transformations: if $x$ is generic over $K \subseteq \mathbb{K}$, $g\in \mathbb{K}^{*}$ and $h\in \mathbb{K}$ such that $x\forkindep^{f}_{K}g,h$ then $g\cdot x + h\forkindep^{f}_{K}g,h$.

\begin{lemma}\label{algindep2} $f$-generic tuples in $\mathbb{K}$ satisfy the following properties:
\begin{enumerate}
    \item Extension: If a tuple $(x_{i})_{i<\kappa}$ is $f$-generic over $K$ and if $(x_{i})_{i<n}\forkindep^{f}_{K}K'$ for $K\subseteq K'$ then $(x_{i})_{i<\kappa}$ is $f$-generic over $K'$.
    \item Genericity: If a finite tuple $(x_{i})_{i<n}$ is $f$-generic over $K$, $M\in GL_{n}(K)$ is an invertible matrix and $(k_{i})_{i<n} \in K$ then the tuple $M \times (x_{i})_{i<n} + (k_{i})_{i<n}$ is also $f$-generic over $K$.
    \item Transitivity: If $\overline{x}$ and $\overline{y}$ are finite tuples, $\overline{x}$ is $f$-generic over $K$ and $\overline{y}$ is $f$-generic over $K(\overline{x})$ if and only if $\overline{x}\overline{y}$ is $f$-generic over $K$.
\end{enumerate}\end{lemma}

\begin{proof}
The first and last points are just transitivity of $\forkindep^{f}$ and \cref{fgener}. For the second point we know that in a simple field genericity for the additive group and the multiplicative group are the same.

\vspace{10pt}
This tells us that if $(\alpha_{i})\in K^{*}$ and $(k_{i})_{i<n}\in K$ the tuple $(\alpha_{i}\cdot x_{i}+k_{i})_{i<n}$ is $f$-generic over $K$. The tuple $(x_{1}+x_{2},x_{2},..,x_{n-1})_{i<n}$ is $f$-generic over $K$ since $x_{1}$ is generic over $Kx_{\centernot =1}$ and $x_{2}$ is generic over $Kx_{\centernot =2}$, and it is clear that with these two transformations and the permutations of the tuple $(x_{i})_{i<n}$ we can generate the transformations $(x_{i})_{i<n} \rightarrow M \times (x_{i})_{i<n} + (k_{i})_{i<n}$ described in the statement.
\end{proof}

\begin{definition}
    Let $E\subseteq A,B \subseteq \mathbb{M}$. We say that $A$ is Gamma-independent of $B$ over $E$, written $A\forkindep^{\Gamma}_{E}B$ if the following conditions holds:
    \begin{enumerate}
        \item  $K(A)\forkindep^{f}_{K(E)}K(B)$.
        \item $\langle A \rangle \cap \langle B \rangle = \langle E \rangle$.
        \item For all $b_{1},..,b_{n} \in V(B)$ linearly independent over $E$ and $a_{1},..,a_{m} \in V(A)$ linearly independent over $E$ the tuple $([a_{i},b_{j}])_{i\leq m,j\leq n}$ is $f$-generic over $K(A)K(B)$.
    \end{enumerate}
\end{definition}

\begin{remark}\label{suffisantgamma}
Using \cref{algindep2} it is easy to show that when $A$ and $B$ have finite dimension over $E$ it is enough to check $3.$ for some basis $b_{1},..,b_{n}$, $a_{1},..,a_{m}$ of $B$ and $A$ over $E$. In fact changing the basis of $A$ over $E$ consists of applying a transformation of the form $(x_{i})_{i<n} \rightarrow M \times (x_{i})_{i<n} + (k_{i})_{i<n}$ with $M\in GL_{n}(K(A))$ and $k_{i} \in \Lin_{K(A)}(V(E))$.

\vspace{10pt}
A similar result holds when the dimension is infinite: If for some basis $(a_{i})_{i\leq \kappa_{A}}$ of $A$ over $E$ and $(b_{i})_{i\leq \kappa_{B}}$ of $B$ over $E$ the tuple $([a_{i},b_{j}])_{i<\kappa_{A},j<\kappa_{B}}$ is $f$-generic over $K(A)K(B)$ then for any basis $(a'_{i})_{i\leq \kappa_{A}}$ of $A$ over $E$ and $(b'_{i})_{i\leq \kappa_{B}}$ of $B$ over $E$ the tuple $([a'_{i},b'_{j}])_{i<\kappa_{A},j<\kappa_{B}}$ is $f$-generic over $K(A)K(B)$.
\end{remark}

\begin{proof} Let $(a_{i})_{i\leq \kappa_{A}}$, $(a'_{i})_{i\leq \kappa_{A}}$, $(b_{i})_{i\leq \kappa_{B}}$ be as in the statement. We begin by showing that any finite subtuple of $([a'_{i},b_{j}])_{i<\kappa_{A},j<\kappa_{B}}$ is $f$-generic over $K(A)K(B)$.

\vspace{10pt}
Consider $(a'_{j'_{k}})_{k<n'}$ for $n'<\omega$, $(j'_{k})_{k<n'}\in \kappa_{A}$ a finite part of $(a'_{i})_{i\leq \kappa_{A}}$. We can find some $n<\omega$ and $(j_{k})_{k<n}\in \kappa_{A}$ such that:

$$\langle E (a'_{j'_{k}})_{k<n'} (a_{j_{k}})_{k<n}\rangle = \langle E (a_{j_{k}})_{k<n}\rangle, $$

and that $(a'_{j'_{k}})_{k<n'}(a_{j_{k}})_{n'\leq k<n}$ is a basis of $\langle E (a_{j_{k}})_{k<n}\rangle $ over $E$. 

\vspace{10pt}
By the finite dimensional case we know that the tuple $([a'_{j'_{k}},b_{l_{k}}])_{k<n',l<m}$ is $f$-generic over $K(A)K(B)$ for any $m<\omega$ and $(l_{k})_{k<m} < \kappa_{B}$. Since this holds for any finite parts of $(a'_{i})_{i\leq \kappa_{A}}$ and $(b_{i})_{i\leq \kappa_{B}}$ by finite character of $\forkindep^{f}$ we get that the tuple $([a'_{i},b_{j}])_{i<\kappa_{A},j<\kappa_{B}}$ is $f$-generic over $K(A)K(B)$.
\end{proof}

\begin{prop}\label{proprietes2} The relation of Gamma-independence is invariant, transitive, symmetric, satisfies existence, monotonicity, finite character and extension.
\end{prop}

\begin{proof} It is clear that $\forkindep^{\Gamma}$ satisfies invariance, monotonicity, existence and symmetry. We begin by showing that $\forkindep^{\Gamma}$ satisfies finite character.

\vspace{10pt}
Assume that for every finitely generated structures over $E$ $A_{0}\subseteq A$ and $B_{0}\subseteq B$ we have $A_{0}\forkindep^{\Gamma}_{E}B_{0}$. Then clearly $\langle A \rangle \cap \langle B \rangle = \langle E \rangle$. $K(A)\forkindep^{f}_{K(E)}K(B)$ follows from the finite character of $\forkindep^{f}$. The other condition of $\forkindep^{\Gamma}$ follows from the same fact: 

\vspace{10pt}
If $b_{1},..,b_{n} \in V(B)$ are linearly independent over $E$ and $a_{1},..,a_{m} \in V(A)$ are linearly independent over $E$ then by assumption the tuple $([a_{i},b_{j}])_{i\leq m,j\leq n}$ is $f$-generic over $K(A_{0})K(B_{0})$ for any $A_{0}\subseteq A$ and $B_{0}\subseteq B$ finitely generated structures over $E$ containing $(a_{i})_{i\leq m}$ and $(b_{j})_{j\leq n}$ respectively. In particular for any finitely generated subfield $K\subseteq K(A)K(B)$ the tuple $([a_{i},b_{j}])_{i\leq m,j\leq n}$ is $f$-generic over $K$, so it is $f$-generic over $K(A)K(B)$ and $A\forkindep^{\Gamma}_{E}B$.

\vspace{10pt}
We now prove the transitivity of $\forkindep^{\Gamma}$. Let $E\subseteq A,B$ and $B\subseteq C$. By finite character we can assume that $A$ is finitely generated over $E$ and that $C$ is finitely generated over $B$. Let $(e_{i})_{i<\kappa_{E}}$ be a basis of $V(E)$, we complete it by $(a_{i})_{i<n}$ and $(b_{i})_{i<\kappa_{B}}$ into some basis of $V(A)$ and $V(B)$ respectively, and complete $(e_{i})_{i<\kappa_{E}}(b_{i})_{i<\kappa_{B}}$ by $(c_{i})_{i<m}$ into a basis of $V(C)$. We assume that $A\forkindep^{\Gamma}_{E}C$ and show that $AB\forkindep^{\Gamma}_{B}C$. Clearly $\langle AB \rangle \cap \langle C \rangle  = \langle B \rangle$.

\vspace{10pt}
By assumption $K(C)\forkindep^{f}_{K(E)}K(A)K(B)$, so by base monotonicity $K(C)\forkindep^{f}_{K(B)}K(A)K(B)$. $K(AB)=K(A)K(B)[a_{i},b_{j}]_{i<n,j<\kappa_{B}}$, and by assumption $[a_{i},b_{j_{k}}]_{i<n,k<m}\forkindep^{K}_{\emptyset}K(A)K(C)$ for any $m<\omega$ and $(j_{k})_{k<m} \in \kappa_{B}$ so, by finite character, $[a_{i},b_{j}]_{i<n,j<\kappa_{B}}\forkindep^{K}_{\emptyset}K(A)K(C)$.

Then using base monotonicity we get: 

$$K(A)K(B)[a_{i},b_{j}]_{i<\kappa_{A},j<\kappa_{B}}\forkindep^{K}_{K(A)K(B)}K(A)K(C),$$

and using transitivity we get that $K(AB)\forkindep^{f}_{K(B)}K(C)$. 

\vspace{10pt}
By assumption the tuples $[a_{i},c_{j}]_{i<n,j<m}[a_{i},b_{j_{k}}]_{i<n,k<m'}$ for any $m<\omega$ and $(j_{k})_{k<m} \in \kappa_{B}$ are $f$-generic over $K(A)K(C)$. By $3.$ of \cref{algindep2} this implies that the tuple $[a_{i},c_{j}]_{i<n,j<m}$ is $f$-generic over $K(A)K(C)[a_{i},b_{j_{k}}]_{i<n,k<m'}$, and finite character entails that $[a_{i},c_{j}]_{i<n,j<m}$ is $f$-generic over $K(AB)K(C)$, so $AB\forkindep^{\Gamma}_{B}C$.

\vspace{10pt}
We now assume that $A\forkindep^{\Gamma}_{E}B$ and $AB\forkindep^{\Gamma}_{B}C$ and we show that $A\forkindep^{\Gamma}_{E}C$. Clearly $\langle A \rangle \cap \langle C \rangle = \langle E \rangle$.

\vspace{10pt}
By assumption $K(AB)\forkindep^{f}_{K(B)}K(C)$ and $K(A)\forkindep^{f}_{K(E)}K(B)$, so $K(A)\forkindep^{f}_{K(E)}K(C)$. By assumption also the tuple $[a_{i},c_{j}]_{i<n,j<m}$ is $f$-generic over $K(AB)K(C)=K(A)K(B)[a_{i},b_{j}]_{i<n,j<\kappa_{B}}K(C)$ and $[a_{i},b_{j}]_{i<n,j<\kappa_{B}}$ is $f$-generic over $K(A)K(B)$, so by \cref{algindep2} the tuple $[a_{i},b_{j}]_{i<n,j<\kappa_{B}}[a_{i},c_{j}]_{i<n,j<m}$ is $f$-generic over $K(A)K(C)$ and $A\forkindep^{\Gamma}_{E}C$.

\vspace{10pt}
We now show that $\forkindep^{\Gamma}$ satisfies extension. Let $E\subseteq A,B$, let $(e_{i})_{i<\kappa_{E}}$ be a basis of $V(E)$, we complete it by $(a_{i})_{i<\kappa_{A}}$ and $(b_{i})_{i<\kappa_{B}}$ into some basis of $V(A)$ and $V(B)$ respectively. Let $\overline{a}_{k}$ be a tuple enumerating $K(A)$. Let $\overline{a}'_{k}\equiv_{K(E)}\overline{a}_{k}$ be such that $\overline{a}'_{k}\forkindep^{f}_{K(E)}K(B)$. We write $\alpha_{i,j}$ for the element of $\overline{a}'_{k}$ corresponding to $[a_{i},a_{j}]\in \overline{a}_{k}$ via this isomorphism, similarly for $\gamma_{i,j}$ and $[a_{i},e_{j}]\in \overline{a}_{k}$.

\vspace{10pt}
Let $(\beta_{i,j})_{i<\kappa_{A},j<\kappa_{B}}$ be an $f$-generic tuple over $\overline{a}'_{k}K(B)$. At this point we have collected all of the elements of the field sort that we want, and we just need to find the vector sort. By \cref{algbili} we can find some vectors $(a'_{i})_{i<\kappa_{A}}$ linearly independent over $\langle B \rangle$ such that:

\begin{enumerate}
    \item[$\cdot$] $[a'_{i},e_{j}]=\gamma_{i,j}$ for all $i<\kappa_{A}$, $j<\kappa_{E}$.
    \item[$\cdot$] $[a'_{i},a'_{j}]=\alpha_{i,j}$ for all $i,j<\kappa_{A}$.
    \item[$\cdot$] $[a'_{i},b_{j}]=\beta_{i,j}$ for all $i<\kappa_{A}$, $j<\kappa_{B}$.
\end{enumerate}

Let $A'$ be the structure defined the following way: $K(A')=\overline{a}'_{k}$ and $(a'_{i})_{i<\kappa_{A}}$ is a basis of $V(A')$ over $V(E)$. Then $A'\equiv_{E}A$ by quantifier elimination (sending $a_{i}$ to $a'_{i}$ and $K(A)$ to $K(A')$ via the isomorphism $\overline{a}_{k}\rightarrow \overline{a}'_{k}$ in the field sort gives us the isomorphism we want). $K(A')\forkindep_{K(E)}K_{B}$ and $\langle A \rangle \cap \langle B \rangle = \langle E \rangle$. The last condition of $\forkindep^{\Gamma}$ follows from \cref{suffisantgamma}.\end{proof}

\begin{prop} Let $E$ be a structure and $p\in S(K(E))$ be a type in the field sort that is an amalgamation basis in $\mathbb{K}$. If $A^{0},A^{1}$ and $B^{0},B^{1}$ are structures containing $E$ such that $\models p(K(A^{0}))$, $A_{0}\equiv_{E}A_{1}$, $A^{i}\forkindep^{\Gamma}_{E}B^{i}$ for $i=0,1$ and $B^{0}\forkindep^{\infty}_{A}B^{1}$. Then there is $A$ such that $A\equiv_{B^{j}}A^{j}$ for $j=0,1$ and $A\forkindep^{\Gamma}_{E}B^{0}B^{1}$. So the relation of Gamma-independence satisfies type amalgamation over models.
\end{prop}

\begin{proof} Let $E$ and $A^{j},B^{j}$ for $j=0,1$ be as in the statement. Let $(e_{i})_{i<\kappa_{E}}$ be a basis of $V(E)$, we complete it by $(a^{0}_{i})_{i<\kappa_{A}}$ and $(b^{j}_{i})_{i<\kappa_{B^{i}}}$ into some basis of $V(A^{0})$ and $V(B^{j})$ for $j=0,1$ respectively. Let $(a^{1}_{i})_{i<\kappa_{A}}$ be the basis of $A^{1}$ over $E$ corresponding to $(a^{0}_{i})_{i<\kappa_{A}}$ trough the isomorphism $A_{0}\equiv_{E}A_{1}$. Let $k_{A}^{0}$ enumerate $K(A^{0})$ and let $k_{A}^{1}$ be the corresponding enumeration of $K(A^{1})$. Since the isomorphism is an isomorphism of $\mathcal{L}$-structures it sends $[a^{0}_{j},a^{0}_{j}]$ to $[a^{1}_{j},a^{1}_{j}]$ for any $i,j<\kappa_{A}$ and similarly for $[a^{0}_{j},e_{j}]$ and $[a^{1}_{j},e_{j}]$.

\vspace{10pt}
Then $k_{A}^{0}\equiv_{K(E)}k_{A}^{1}$, $k_{A}^{j}\forkindep^{f}_{K(E)}K(B^{j})$ for $j=0,1$ and $K(B^{0})\forkindep^{f}_{K(E)}K(B^{1})$. By type amalgamation in $T_{K}$ there is $k\models \tp(k_{A}^{0}/K(B^{0}))\cup \tp(k_{A}^{1}/K(B^{1}))$ such that $k\forkindep^{f}_{K(E)}K(B^{0})K(B_{1})$. By extension we can assume that $k\forkindep^{f}_{K(E)}K(B^{0}B^{1})$.

\vspace{10pt}
By assumption the tuple $([a^{0}_{i},b^{0}_{l}])_{i<\kappa_{A},j<\kappa_{B^{0}}}$ is generic over $K(A^{0})K(B^{0})$. By extension there is $(\beta^{0}_{i,j})_{i<\kappa_{A},j<\kappa_{B^{0}}}$ such that $k(\beta^{0}_{i,j})_{i<\kappa_{A},j<\kappa_{B^{0}}} \equiv_{K(B^{0})}K(A^{0})([a^{0}_{i},b^{0}_{l}])_{i<\kappa_{A},j<\kappa_{B^{0}}}$ and $(\beta^{0}_{i,j})_{i<\kappa_{A},j<\kappa_{B^{0}}}$ is generic over $k'K(B^{0}B^{1})$.

\vspace{10pt}
By the same argument there is a tuple $(\beta^{1}_{i,j})_{i<\kappa_{A},j<\kappa_{B^{1}}}$ such that: 

$$k(\beta^{1}_{i,j})_{i<\kappa_{A},j<\kappa_{B^{1}}} \equiv_{K(B^{1})}K(A^{1})([a^{1}_{i},b^{1}_{l}])_{i<\kappa_{A},j<\kappa_{B^{0}}},$$

and $(\beta^{1}_{i,j})_{i<\kappa_{A},j<\kappa_{B^{1}}}$ is $f$-generic over $k'K(B^{0}B^{1})(\beta^{0}_{i,j})_{i<\kappa_{A},j<\kappa_{B^{0}}}$. Then by \cref{algindep2} the tuple $(\beta^{0}_{i,j})_{i<\kappa_{A},j<\kappa_{B^{0}}}(\beta^{1}_{i,j})_{i<\kappa_{A},j<\kappa_{B^{1}}}$ is generic over  $kK(B^{0}B^{1})$. Let $\gamma_{i,j}$ be the element of $k$ corresponding to $[a^{0}_{i},e_{j}]$ for every $i<\kappa_{A},j<\kappa_{E}$ and $\alpha_{i,j}$ the one corresponding to $[a^{0}_{i},a^{0}_{j}]$ for every $i<\kappa_{A},j<\kappa_{A}$ trough the isomorphism $k\equiv_{K(B^{0})}k_{A}^{0}$ (notice that using $A^{1}$ instead gives the same elements of $k$).

\vspace{10pt}
By \cref{algbili} we can find some vectors $(a_{i})_{i<\kappa_{A}}$ linearly independent over $\langle B^{0}B^{1} \rangle$ such that:

\begin{enumerate}
    \item[$\cdot$] $[a_{i},e_{j}]=\gamma_{i,j}$ for all $i<\kappa_{A}$, $j<\kappa_{E}$.
    \item[$\cdot$] $[a_{i},a_{j}]=\alpha_{i,j}$ for all $i,j<\kappa_{A}$.
    \item[$\cdot$] $[a_{i},b^{0}_{j}]=\beta^{0}_{i,j}$ for all $i\kappa_{A}$, $j<\kappa_{B^{0}}$.
    \item[$\cdot$] $[a_{i},b^{1}_{j}]=\beta^{1}_{i,j}$ for all $i\kappa_{A}$, $j<\kappa_{B^{1}}$.
\end{enumerate}

Let $A$ be the structure defined by $K(A)=k$ and $(a_{i})_{i<\kappa_{A}}$ is a basis of $V(A)$ over $V(E)$. Then $A\equiv_{B^{i}}A^{i}$ for $i=0,1$ by quantifier elimination (sending $a_{j}$ to $a^{i}_{j}$ and $K(AB^{i})$ to $K(A^{i}B^{i})$ via the isomorphism $k(\beta^{i}_{j,k})_{j<\kappa_{A},k<\kappa_{B^{i}p}} \equiv_{K(B)} k_{A}^{i}([a^{i}_{j},b^{i}_{l}])_{j<\kappa_{A},l<\kappa_{B^{i}}}$ in the field sort gives us the isomorphism we want). $K(A)\forkindep_{K(E)}K_{B}$ and $\langle A \rangle \cap \langle B \rangle = \langle E \rangle$. The fact that $A\forkindep^{\Gamma}_{E}B^{0}B^{1}$ follows from \cref{suffisantgamma}.\end{proof}

\begin{prop}\label{preuveexist} $A\forkindep^{\Gamma}_{E}B$ implies $D\forkindep^{f}_{A}B$ for all $E \subseteq A,B$. So the theory $sT^{K}_{\infty}$ satisfies existence.
\end{prop}

\begin{proof} By finite character it is enough to prove this for $V(A)$ and $V(B)$ of finite dimension over $V(E)$. Since $\forkindep^{\Gamma}$ satisfies extension it is enough to show that $A\forkindep^{\Gamma}_{E}B$ implies $D\forkindep^{d}_{A}B$. Let $(e_{i})_{i<\kappa_{E}}$ be a basis of $V(E)$, we complete it by $(a_{i})_{i<n}$ into a basis of $V(A)$. Let $(B_{i})_{i<\omega}$ be an $E$-indiscernible sequence with $B_{0}=B$. We want to find an $A'$ such that $A'B_{i}\equiv_{E}AB_{0}$ for all $i<\omega$.

\vspace{10pt}
We find a negative part and extend the sequence to a sequence $(B_{i})_{i\in \mathbb{Z}}$ indexed on the relative integers. By extension of $\forkindep^{\Gamma}$ we can assume that $A\forkindep^{\Gamma}_{E}B_{\mathbb{Z}}$, where $B_{\mathbb{Z}}$ is the structure generated by $\lbrace B_{i}$ : $i\in \mathbb{Z}\rbrace$. We will write $B_{[i,j]}$ for the structure generated by $B_{i}..B_{j}$ for every $i<j\in \mathbb{Z}$. $A\forkindep^{\Gamma}_{EB_{<0}}B_{\mathbb{Z}}$ by base monotonicity of $\forkindep^{\Gamma}$. We define $E':=B_{<0}$, by indiscernibility the $B_{i}$ for $i\in \omega$ are linearly disjoint over $E'$ and they form an $E'$ indiscernible sequence. In fact if some vectors $b_{i}\in B_{i}$ is in $\langle B_{<i}\rangle$ then by indiscernibility it is in $\langle B_{<0}\rangle$.

\vspace{10pt}
Thanks to this we can assume without loss of generality (eventually replacing $E$ by $E'$) that the sequence $(B_{i})_{i<\omega}$ is linearly disjoint over $E$. What happens here is that the elements of $(B_{i})_{i<\omega}$ might overlap outside of the original $\langle E \rangle$ but since the property of being linearly disjoint has local character, by indiscernibility, all the overlapping that might happen is already happening inside of the negative part of the sequence. Let $(b^{0}_{i})_{i<n}$ be a basis of $V(B^{0})$ over $V(E)$, and let $(b^{i}_{i})_{j<n}$ be the corresponding basis of $V(B^{i})$. 

\vspace{10pt}
Let $k$ be a tuple enumerating $K(A)$ and $p(x,y,z)=\tp(k,([a_{i},b_{j}^{0}])_{i<n,j<m},K(B_{0})/K(E))$. By assumption $k([a,b_{j}^{0}])_{j<n}\forkindep^{f}_{K(E)}K(B_{0})$, so there is a tuple $k',(\beta_{i,j})_{i<n,j<m}$ such that $$k',(\beta_{i,j})_{i<n,j<m}\models \bigcup\limits_{i<\omega}p(x,y,K(B_{i})).$$

Write $\gamma_{i,j}$ the element of $k'$ corresponding to $[a_{i},e_{j}] \in k$ for all $i<n$, $j<\kappa_{E}$ and $\alpha_{i,j}$ the element of $k'$ corresponding to $[a_{i},a_{j}] \in k$ for all $i,j<n$. By \cref{algbili} we can find some vectors $(a'_{i})_{i<n}$ linearly independent over $\langle B_{<\omega} \rangle$ such that:

\begin{enumerate}
    \item[$\cdot$] $[a'_{i},e_{j}]=\gamma_{i,j}$ for all $i<\kappa_{A}$, $j<\kappa_{E}$.
    \item[$\cdot$] $[a'_{i},a'_{j}]=\alpha_{i,j}$ for all $i,j<\kappa_{A}$.
    \item[$\cdot$] $[a'_{i},b^{j}_{l}]=\beta_{i,j}$ for all $i<n$, $j<\omega$ and $l<m$.
\end{enumerate}

Let $A'$ be the structure defined by $K(A')=k'$ and $V(A')=\Lin_{k'}((a'_{i})_{i<n})$. By quantifier elimination $A'B_{i}\equiv_{E}AB_{0}$ for all $i<\omega$, which concludes our proof.\end{proof}

\section{Forcing Base Monotonicity of Kim-Independence}

In Chapter 2 we gave some example of NSOP$_1$ theories such that $\forkindep^{K^{M}}=\forkindep^{f}$. It was conjectured that this holds in any NSOP$_1$ theory. A counter-example was given in \cite{bossut2023note} in $sT^{ACF}_{\infty}$.

\vspace{10pt}
This counter-example is partial in the following sense: $sT^{ACF}_{\infty}$ does not eliminate imaginaries, and in the imaginary extension $sT^{ACF,eq}_{\infty}$ this example is not a counter-example anymore. 

\vspace{10pt}
We begin by presenting the said counter example and then we fill this gap by showing that $\forkindep^{K^{M}}\centernot= \forkindep^{f}$ in $sT^{ACF,eq}_{\infty}$. For this we use a result of weak elimination of imaginaries of Dobrowolski to characterize $sT^{ACF,eq}_{\infty}$. We will write $\forkindep^{ACF}$ for forking independence in $ACF$.

\subsection{A counter-example for the home sorts}



Consider $E \models sT^{ACF}_{\infty}$ a small model and let $(e_{i})_{i<\omega}$ be a $K(E)$-basis of $V(E)$. Let $a,b_{0},b_{1}$ be vectors linearly independent over $\langle E\rangle$ such that $a,b_{0},b_{1} \in E^{\bot}$, $[a,a]=[b_{0},b_{0}]=[b_{1},b_{1}]=0$ and $[a,b_{0}]=[a,b_{1}]=[b_{0},b_{1}]=1$.

\vspace{10pt}
Let $A=(\Lin_{K(E)}(V(E),a),K(E))$ and $B=(\Lin_{K(E)}(V(E),b_{0},b_{1}),K(E))$. Then $B$ is a model of $sT_{\infty}^{ACF}$ and $K(AB)=K(E)$ where $K(AB)$ is the field sort of the structure generated by $A$ and $B$. Both $A$ and $B$ are algebraically closed and contain $E$, and by definition we have $A \forkindep^{K}_{E}B$. If $B_{0}$ is such that $E \subseteq B_{0} \subseteq B$ then $B_{0}$ can be $E$ and $B$, and in these two cases $AB_{0}\forkindep^{K}_{B_{0}}B$ holds trivially.

\vspace{10pt}
Otherwise $B_{0}=Eb$ is generated over $E$ by a single vector $b\in B$. In that case this vector $b$ can be written as $b= \lambda_{0}\cdot b_{0} + \lambda_{1}\cdot b_{1}$ for $\lambda_{0},\lambda_{1}\in K(E)$. Then $AB_{0}=(\Lin_{K(E)}(V(E),a,b),K(E))$ and $AB_{0}\forkindep^{K}_{B_{0}}B$. In all of these cases $AB_{0}\forkindep^{K}_{B_{0}}B$, so $A\forkindep^{K^{M}}_{E}B$.

\vspace{10pt}
Now let us consider some $\alpha \in \mathbb{K}$ transcendental over $K(E)$ and $B':=(\Lin_{K(E)[\alpha]}(V(E),b_{0},b_{1}),K(E)[\alpha])$. We show that there is no $A'\equiv_{B} A$ such that $A'\forkindep^{K^{M}}_{E}B'$, which implies that $\forkindep^{K^{M}}$ does not satisfy extension.

\vspace{10pt}
If there were such an $A'$, setting $B_{0}'=(\Lin_{K(E)}(V(E),\alpha \cdot b_{0},\alpha^{-1} \cdot b_{1}),K(E))$ we would have $A'B_{0}'\forkindep^{K}_{B_{0}'}B'$, so $[a',\alpha \cdot b_{0}]\forkindep^{f}_{K(E)}\alpha$ and $\alpha \forkindep^{f}_{K(E)}\alpha$ so $\alpha\in K(E)$, a contradiction. This shows us that in an NSOP$_1$ theory with existence $\forkindep^{K^{M}}$ is not necessary equal to $\forkindep^{f}$ over models.

\vspace{10pt}
As we mentioned $sT^{ACF}_{\infty}$ does not weakly eliminate imaginaries: In fact the equivalence relation between $n$-tuples of vector defined by $E(\overline{x},\overline{y})$ if and only if $\overline{x}$ and $\overline{y}$ generate the same vector space is not eliminated. We will write $\langle \overline{u}\rangle$ for the $E$-class of $\overline{u}$. If we consider Kim-independence in the imaginary extension then $A\forkindep_{E}^{K^{M}}B$ does not hold for the previous $E,A,B$:

\vspace{10pt}
Let $\langle b_{0}\rangle $ be the line generated by $b_{0}$ seen as an imaginary. Then $A\langle b_{0}\rangle \centernot \forkindep^{K}_{E\langle b_{0}\rangle}B$ since $b_{0}\in acl(A,\langle b_{0}\rangle )$ and $b_{0}\centernot\in acl(E,\langle b_{0}\rangle)$: In fact $b_{0}$ is the only point $x$ of the line $\langle b_{0}\rangle$ that satisfies $[a,x]=1$. 

\subsection{A counter-example for $sT^{ACF,eq}_{\infty}$}

We consider the imaginary expansion $sT^{ACF,eq}_{\infty}$ of $sT^{ACF}_{\infty}$. 
As we mentioned in the previous subsection $sT^{ACF}_{\infty}$ does not weakly eliminate the equivalence classes of the relations `generating the same vector subspace'. We add additional sorts for the equivalence classes of these relations and write $sT^{ACF}_{\infty, Gr}$ for this theory in the extended language.

\vspace{10pt}
The following result is due to Dobrowolski.

\begin{prop}
$sT^{ACF}_{\infty, Gr}$ has weak elimination of imaginaries.
\end{prop}

\begin{proof}
The proof was kindly communicated to me by Dobrowolski (personal communication).
\end{proof}

We now show that $\forkindep^{K^{M}}$ does not satisfy extension in $sT^{ACF,eq}_{\infty}$, which yields that $\forkindep^{K^{M}}\centernot=\forkindep^{f}$.

\vspace{10pt}
Let $a,b_{0}\in \mathbb{V}$ and $\alpha\in \mathbb{K}$ be such that $[b_{0},b_{0}]=1$, $[a,a]=0$, $[a,b_{0}]=\alpha$ and $\alpha$ is transcendental over the prime field.

\vspace{10pt}
It is easy to see that $\alpha, a\forkindep^{K^{M}}_{\emptyset}b_{0}$, in fact by weak elimination of imaginaries all of the elements in $acl^{eq}(b_{0})\setminus acl^{eq}(\emptyset)$ are interalgebraic with $b_{0}$ or $\langle b_{0}\rangle$, and these two elements are interalgebraic since $b_{0}$ is one of the two elements $x \in \langle b_{0}\rangle$ such that $[x,x]=1$.

\vspace{10pt}
Let $b_{1}$ be such that $b_{1}\centernot\in \langle b_{0}\rangle$, $[b_{1},b_{1}]=0$ and $[b_{0},b_{1}]=0$. We show that there is no $\alpha',a'\equiv_{b_{0}}\alpha,a$ such that $\alpha',a'\forkindep^{K^{M}}_{\emptyset}b_{0}b_{1}$. Otherwise let $\alpha',a'$ be such a tuple.

\vspace{10pt}
If $[a',b_{1}]=0$, then $\alpha' a'\centernot\forkindep^{K}_{\langle b_{0}b_{1} \rangle}b_{0}b_{1}$: The set $b_{0}+\langle b_{1}\rangle$ can be seen as the class of $(b_{0},b_{1})$ for the relation $(x,y)E(x',y'):= (\langle y \rangle = \langle y'\rangle)\wedge (x-x'\in \langle y \rangle)$. This set is definable over $a'\alpha'\langle b_{0}b_{1}\rangle$ as $\lbrace x $ : $x\in \langle b_{0}b_{1}\rangle$ and $[x,a']=\alpha'\rbrace$, so as an imaginary it is in the algebraic closure of $a'\alpha'\langle b_{0}b_{1}\rangle$. However $b_{0}+\langle b_{1}\rangle$ is not in the algebraic closure of $\langle b_{0}b_{1}\rangle$ since $b_{0}+\langle b_{1}\rangle\equiv_{\langle b_{0}b_{1}\rangle}(b_{0} +\lambda \cdot b_{1})+\langle b_{1}\rangle$ for all $\lambda \in \mathbb{K}$.

\vspace{10pt}
Otherwise assume that $[a',b_{1}]=\beta' \centernot=0$. We show that $\alpha',a'\centernot\forkindep^{K}_{\langle b_{0}b_{1} \rangle,b_{1}}b_{0}b_{1}$. Then $a'^{\perp}\cap \langle b_{0}b_{1}\rangle = \langle \beta'\cdot b_{0} - \alpha'\cdot b_{1}\rangle$. We can define the set of the points $x$ in this line that satisfy $[x,x]=1$. This set is finite and contains $b_{0}+\frac{\alpha'}{\beta'}\cdot b_{1}$. Since $\alpha',\beta',b_{1}\in acl(\alpha'a'\langle b_{0}b_{1}\rangle b_{1})$ we deduce that $b_{0}\in acl(\alpha'a'\langle b_{0}b_{1}\rangle b_{1})$. However $b_{0}$ is not in the algebraic closure of $\langle b_{0}b_{1} \rangle b_{1}$. In fact $b_{0}\equiv_{\langle b_{0}b_{1}\rangle b_{1}}b_{0} + \lambda \cdot b_{1}$ for all $\lambda\in \mathbb{K}$.

\begin{remark} This theory is to my knowledge the only example of NSOP$_1$ theory with existence that does not satisfy that $\forkindep^{f}=\forkindep^{K^{M}}$. I should also mention that there is not yet a characterization of forking independence in this theory.
\end{remark}

\bibliographystyle{plain}
\bibliography{ref.bib}

\end{document}